\documentclass[a4paper, 11pt]{amsart}
\usepackage{newtxtext,newtxmath} 
\usepackage[left=3cm,right=2cm,top=2.5cm,bottom=2.5cm]{geometry}
\usepackage[utf8]{inputenc}
\usepackage{stackrel,verbatim,textcomp}
\usepackage[T1]{fontenc}
\usepackage[pdftex]{graphicx}
\usepackage{upgreek, epstopdf,dsfont,soul}
\usepackage{mathtools,listings,marginnote}
\usepackage{xargs} 
\usepackage{tikz}
\usetikzlibrary{shapes,arrows,calc,positioning,patterns}
\usepackage{pgfplots}
\usepackage[colorlinks=true,linkcolor=red,citecolor=blue]{hyperref}
\usepackage{caption}
\usepackage{subcaption}
\usepackage{orcidlink}
\usepackage[normalem]{ulem}
\usepackage[textsize=tiny]{todonotes}

\pgfplotsset{compat=1.17}

\renewcommand{\leq}{\leqslant}
\renewcommand{\geq}{\geqslant}

\newcommand{\ind}{\mathbf{1}}
\renewcommand{\P}[1]{\mathbb{P}\left(#1\right)}
\newcommand{\e}{\varepsilon}
\newcommand{\Z}{{\mathbb{Z}}}

\newcommand{\R}{\mathbb{R}}

\newcommand{\Hprod}[1]{\left\langle #1\right\rangle}
\newcommand{\Ex}[1]{\mathbb{E}\left(#1\right)}
\newcommand{\Var}[1]{\mathrm{Var}\left(#1\right)}

\newcommand{\hatpi}{\widehat{\pi}}
\theoremstyle{plain}
\newtheorem{theorem}{Theorem}
\newtheorem{proposition}[theorem]{Proposition}

\newtheorem{lemma}[theorem]{Lemma}

\theoremstyle{definition}

\begin{document}
\title{Hydrodynamic Limit for repeated averages on the complete graph} 
\author{Alberto M. Campos\orcidlink{0000-0002-1823-0838}}
\author{Tertuliano Franco \orcidlink{0000-0002-1549-2875}}
\address{Universidade Federal da Bahia, Instituto de Matemática, Campus de Ondina, Av. Adhemar de Barros, S/N, CEP 40170-110, Salvador/BA,
Brazil}
\email{tertu@ufba.br}
\author{Markus Heydenreich\orcidlink{0000-0002-3749-7431}}
\author{Marcel Schrocke\orcidlink{0009-0004-6528-6035}}
\address{Universität Augsburg, Department of Mathematics, D-86135 Augsburg, Germany}
\email{albertomizrahycampos@gmail.com, markus.heydenreich@uni-a.de, marcel.schrocke@uni-a.de}

\begin{abstract}
We establish a hydrodynamical limit for the averaging process on the complete graph with $N$ vertices, showing that, after a timescale of order $N$, the empirical distribution of opinions converges to a unique measure. Moreover, if the initial distribution is absolutely continuous with respect to the Lebesgue measure, the limiting measure remains absolutely continuous and its density satisfies a non-diffusive differential equation, that resembles the Smoluchowski coagulation equation.
\end{abstract}

\date{\today}

\subjclass{ 60K35, 
82C22,  	
 35Q70,  	
 37L05,  	
 47J35
 }

\maketitle


\section{Introduction} \label{Sec:Intro}

Consider a set of $N$ individuals, each with an initial opinion $\omega_0 = (\omega_0(1), \dots, \omega_0(N)) \in \mathbb{R}^N$. We define the averaging process as a Markov chain $(\omega_k)_{k \in \mathbb{N}}$ where, given a configuration $\omega_k$, the configuration $\omega_{k+1}$ is obtained by selecting two individuals, $x$ and $y$, uniformly at random and updating  their opinions to their mean value, $\omega_{k+1}(x) = \omega_{k+1}(y) = \frac{\omega_k(x) + \omega_k(y)}{2}$, while all other opinions remain unchanged.

First introduced by Deffuant et al.\ \cite{Deffuant2000}, each individual in the averaging process has their opinion converging towards the initial average $\frac{1}{N}\sum_{k=1}^N \omega_0(k)$. Furthermore, Chatterjee et al.\ \cite{Chatterjee2022} proved that the averaging process has mixing time $\frac{N\ln{N}}{2\ln{2}}$  with a cut-off window of order $N\sqrt{\ln{N}}$. 
In this article, we contribute to this picture by showing a hydrodynamic limit on the density of opinions that occurs on a time scale of order $N$, and where the density of opinions converges to the solution of a differential equation that resembles the Smoluchowski coagulation equation. 

\subsection*{Main result} To state the result formally, fix an integer $N \geq 1$ and consider the complete graph with vertex set $[N] = \{1, 2, \ldots, N\}$. Consider the state space $\Omega_N=\R^N$, with configuration written as $\omega=(\omega(x))_{x\in [N]} \in \R^N$,
and set the averaging process as a continuous time Markov process with a generator given by
\begin{align}\label{eq:GeneratorClick}
    \mathcal{L}_{N}f(\omega)= \frac{1}{N} \sum_{x,y\in [N]} \left[f(A_{x}^y\omega)-f(\omega)\right],\quad \omega\in \Omega_N,
\end{align} where $f$ is any test function. For $x,y,z \in [N] $, the operator $ A_x^y $ acts on the configuration as  
\begin{align}\label{eq:operatorAxy}
A_x^y \omega(z) = 
\begin{cases} 
\omega(z), &\text{if } z\notin\{x,y\};\vspace{4pt}\\
\dfrac{\omega(x)+\omega(y)}{2},  &\text{if } z\in\{x,y\}.
\end{cases}
\end{align}
We further define the empirical measure associated with the averaging process $(\omega_t)_{t\geq 0}$.
\begin{align}\label{eq:EmpiralMeasureintheclick}
    \pi^N_t(du) =\pi^N(\omega_t,du)=\frac{1}{N}\sum_{x\in [N]} \delta_{\omega_t(x)} (du).
\end{align} Our main result is the following. 
\begin{theorem}\label{thm:HydrodinamiclimitClick}
 Assuming that $\omega_0 = (\nu_i)_{i \in [N]} \in \Omega_N$, where $\nu_i$ are i.i.d. random variables with distribution $\nu$ with compact support in the interval $[-M,M]$. For any time $T>0$, the sequence of trajectories of random measures $\{\pi_t^N : 0 \leq t \leq T\}$ converges weakly to a measure trajectory $\{\pi_t : 0 \leq t \leq T\}$, where $\pi_t$ is the unique measure that satisfy a differential measure equation (see \eqref{eq:measureequation}). Additionally, if $\nu(dx)=\rho_0(x)\,dx$ is absolutely continuous (w.r.t.\ Lebesgue measure), then $\pi_t$ is also absolutely continuous with density $\rho(x,t)$ that is the almost everywhere unique solution of the differential equation 
\begin{align}\label{eq:convolution-differential-equation}
  \begin{cases}
    \partial_t \rho(t, u) = 4\left[\rho(t, \cdot) \ast \rho(t, \cdot) \right](2u) - 2\rho(t, u), &(t,u)\in [0,\infty)\times \R; \\
    \rho(0, \cdot) = \rho_0(\cdot).
  \end{cases}
\end{align}
\end{theorem}

As an application of the theorem, Figures \ref{fig:SimulationBer} and \ref{fig:SimulationCont} illustrate two different simulations of the density $\pi_t^N$ in the averaging process.  In Figure \ref{fig:SimulationBer}, the initial distribution of opinions follows $\mathrm{Ber}(1/2)$ and in each time $t$ the density converges to a discrete measure that assigns positive probability to the dyadic numbers $\left\{0,1,\frac{1}{2}, \frac{1}{4}, \frac{3}{4}, \dots \right\}$, which correspond to the possible trajectories of opinions that change a finite number of times. In contrast, Figure \ref{fig:SimulationCont} starts with an initial density given by $\nu(dx) = \ind\{x \in [0,1]\} 2x \,dx$, and its evolution follows the solution of the differential equation~\eqref{eq:convolution-differential-equation}.

\begin{figure}[h!]
    \centering
    \includegraphics[width=1\linewidth]{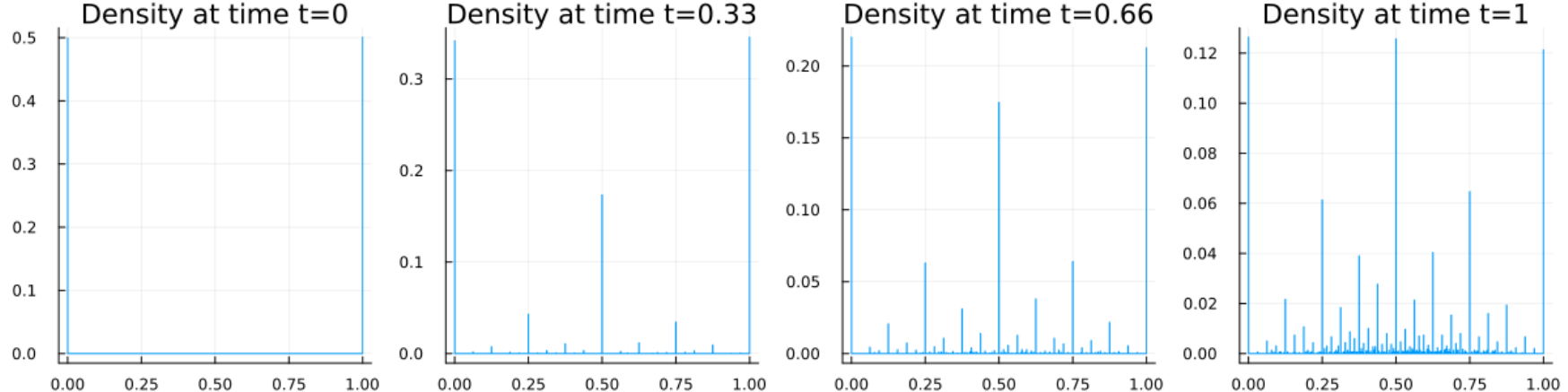}
    \caption{Simulation of the evolution of the opinion density, starting with an initial distribution $\nu \sim \mathrm{Ber}(1/2)$ and considering $N = 3000$ opinions.}
    \label{fig:SimulationBer}
\end{figure}

\begin{figure}[h!]
    \centering
    \includegraphics[width=0.5\linewidth]{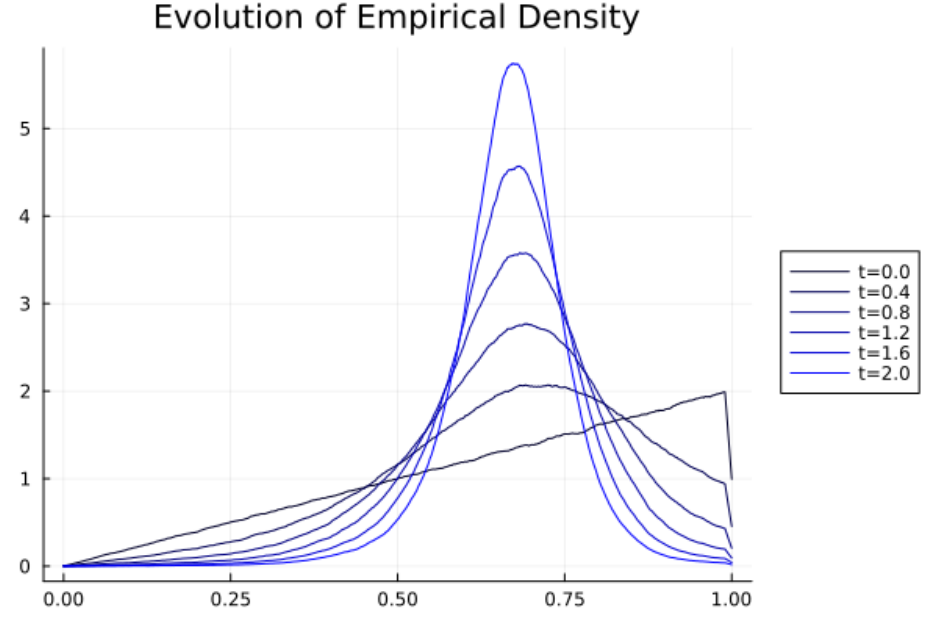}
    \caption{Simulation of the evolution of the opinion density, starting with an initial absolutely continuous distribution $\nu$ such that $\P{\nu<t}=\int_0^t 2x dx$, for $t\in [0,1]$, and considering $N=10^6$ opinions.}
    \label{fig:SimulationCont}
\end{figure}

\subsection*{Background and remarks}  


The differential limit equation \eqref{eq:convolution-differential-equation}, as far as we know, is new in the literature. However, it bears a clear resemblance to the Smoluchowski coagulation equation, which describes the evolution of the density of particles that undergo a coagulation process, see  \cite{Melzak1957,Smoluchowski1927}. In the general Smoluchowski equation, the convolution term is typically weighted by a kernel. In contrast, equation \eqref{eq:convolution-differential-equation} introduces a significant analytical distinction: the convolution term is accelerated by the value $2u$. 

A different connection arises when we view trajectories of opinions as random walks that depend on their density, so that  \eqref{eq:convolution-differential-equation} resembles the nonlinear Fokker-Planck equation for the McKean-Vlasov process (see Chapter 1 and 5 of \cite{Risken1996}). The difference between the limit equations can be explained by the importance of the of non-local interactions for the averaging process.

The solution $\rho(t,x)$ of equation \eqref{eq:convolution-differential-equation} exists and is unique in the Sobolev space $W^{1,1}([0,T])\cap L^1([-M,M])$. This uniquely determines the empirical measure $\pi_t(A) = \int_A \rho(t, x) \, dx$. In contrast, the question of existence and uniqueness of the Smoluchowski equation is only solved for certain choices of kernels; see \cite{Norris1999,McLeod1962}.

Continuing the discussion on the differential equation, we found a nontrivial solution that does not fall within the scope of the theorem. Starting from the Cauchy distribution with parameter $a>0$, the function $$f(x, t) = \frac{1}{\pi} \frac{a}{a^2+x^2} \frac{1}{2e^{2t}-1}$$ is indeed a solution of \eqref{eq:convolution-differential-equation}. This non-trivial solution poses a question about gelation properties that the averaging process may present; see \cite{Hendriks1983} for an analytical perspective and \cite{https://doi.org/10.48550/arxiv.2401.06668} for a probabilistic approach.  In this case, the distribution lacks a first moment and does not satisfy the theorem hypothesis, but the solution presents an exponential loss of mass, that might correspond to the solution of the averaging process in this case.

Regarding the hydrodynamic limit, there are several results in the literature that associate the dynamics of opinions to the solution of differential equations. For instance Cox and Griffeath \cite{Cox1986} demonstrate that the density of opinions in the two-dimensional voter model satisfies the Fisher-Wright diffusion. Sau \cite[Corollary 2.6]{Sau2024}  shows that the opinion model on the $d-$dimensional torus at a fixed time $t$ closely approximates the solution of the heat equation. Similarly, Tomé et al.\  \cite{Tom2023} study variations of the voter model and demonstrate that the evolution of opinions follows a thermodynamic equation. Aldous and Lanoue \cite{Aldous2012} proved that the mean of the process evolves following a differential equation.   It is also worth mentioning that in the work of Deffuant et al.~\cite{Deffuant2000}, a differential equation is used to describe the behavior of opinions on the complete graph with range dependencies.

The averaging process can be defined on any locally-finite (finite or infinite) graph. Recently, Gantert and Vilkas \cite{Nina2024} proved that the averaging process on bounded degree graphs converges for i.i.d.\ initial distribution with suitable moment conditions. 
 
For sake of completeness and to highlight the differences of our main theorem with the current literature, consider the averaging process on the $d$-dimensional hypercubic lattices. That is, fix integers $N\geq 1$ and $d \geq 1$. Define the $d$-dimensional discrete torus of size $N$ as $\mathbb{T}^d_N=(\mathbb{V}_N^d, \mathbb{E}_N^d)$, where the vertex set is $\mathbb{V}_N^d = \left(\mathbb{Z}/N\mathbb{Z}\right)^d$, and the set of unoriented edges is $\mathbb{E}_N^d = \{ \{x, y\} \subseteq \mathbb{V}_N^d \times  \mathbb{V}_N^d : |x - y| = 1 \}$. 
In the $d-$dimensional torus $\mathbb{T}^d_N$, consider $\Omega_N^d=\{(\omega(x))_{x\in \mathbb{V}^d_N}: \omega(x)\in \R\}$, and write the generator of the averaging process, as
\begin{align*}
    \mathcal{L}_{\mathbb{T}^d_N} 
    f(\omega)= N^{2} \sum_{e=\{x, y\}\in \mathbb{E}_N^d} \left[f(A_x^y\omega)-f(\omega)\right],\quad  \omega\in \Omega_N^d, 
\end{align*}  
where $f$ is any test function, and $A_{x}^y$ defined in \eqref{eq:operatorAxy}. The weighted empirical measure
\begin{align}\label{eq:EmpiricalmeasureinTd}
\hatpi_t^N(du)=\hatpi^N(\omega_t,du)=\frac{1}{N^d} \sum_{x\in \mathbb{T}_N^d} \omega_t(x) \delta_{x/N}(du).
\end{align} Then, the averaging process in the $d-$dimensional lattice satisfies:
\begin{theorem}\label{Thm:Hydrodinamics-Z^d}
 Let $\rho_0 \colon \mathbb{T}^d = (\R/\Z)^d \to [-M, M] \subset \R$ be an initial density profile. For every $x \in \mathbb{V}_N^d$, define the initial opinion as $\omega_0^N(x) = \rho_0\left(\frac{x}{N}\right)$. Then, for any time $T > 0$, the sequence of random measures $\{\hatpi_t^N \colon 0 \leq t \leq T\}$ converges weakly as trajectories to the absolutely continuous trajectory $\{\hatpi_t(du) = \rho(t, u) \, du : 0 \leq t \leq T\}$, where the density $\rho(t, u)$ is the unique solution of the heat equation
\begin{align*}
  \begin{cases}
    \partial_t \rho(t,u) = \frac{1}{2}\Delta \rho(t,u), \quad (t,u)\in [0,\infty)\times \mathbb{T}^d;\\
    \rho(0, \cdot) = \rho_0(\cdot).
  \end{cases}
\end{align*}
\end{theorem}

The proof of this theorem follows, with no significant deviations, the steps of the hydrodynamic limit proof for the simple symmetric exclusion process, as presented in Chapter 4 of Kipnis and Landim \cite{Kipnis1999}, and is omitted here. 

We only state it to make evident the differences between the $d-$dimensional lattice and the complete graph. For instance, the definition of the weighted empirical measure \eqref{eq:EmpiricalmeasureinTd} differs from the empirical measure \eqref{eq:EmpiralMeasureintheclick} introduced in Theorem \ref{thm:HydrodinamiclimitClick}. The presence of non-local interactions in the complete graph makes the convergence of the measure $\hatpi^N$ in the Feller-Dynkin martingale complex. The measure $\pi^N$ also poses difficulties, and indeed makes the problem converge to solutions of a measure equation, not a classical differential equation. For further discussion on this issue, see Proposition \ref{Prop:Abscontinuous}.

It is also interesting to point out that by modifying the definition, the qualitative behavior of the limiting equations changes. While in a $d$-dimensional lattice the system follows the heat equation, which is a diffusive equation, we obtain in \eqref{eq:convolution-differential-equation} a non-diffusive equation. In fact, it exhibits the opposite behavior: the mass concentrates at a single point, as can be observed from the definition of the empirical measure in ~\eqref{eq:EmpiralMeasureintheclick}.

For further results on the hydrodynamic limit in $d$-dimensional lattices, we refer to Sau's paper~\cite{Sau2024}. There, using an argument from Caputo, Quattropani, and Sau  \cite{Quattropani2023}, the author relates the spread of the random walk to the spread of information in the averaging process, providing quantitative bounds on the convergence of the averaging process to the heat equation.  Moreover, similar developments are being done for random averaging models, see ~\cite{kipnis1982heat}. 

\subsection*{Definitions} 
On the real line $\R$, consider the Borel $\sigma$-algebra and the usual Lebesgue measure $\lambda$. For two measures $\nu, \mu$, we write $\nu \ll \mu$ whenever $\nu$ is absolutely continuous with respect to $\mu$.  Let $(\mathcal{S},\mathcal{F})$ be a space with its $\sigma$-algebra, denote by $\mathcal{M}_+^1(\mathcal{S})$ the space of probability measures on $\mathcal{S}$ and set $\mathcal{M}(\mathcal{S})$ the space of measures, both spaces equipped with the weak topology. For any point $x \in \mathbb{R}^d$, denote by $\delta_x(\cdot)$ the Dirac measure that assigns total measure to the point $x$.

Regarding function spaces, we equip the product space $[0,T]\times \R$ with the Borel $\sigma$-algebra, and denote by $C^{m,n}( [0,T] \times \R)$ the space of Borel-measureable functions $f \colon [0,T] \times \R \to \R$ with $m$ continuous derivatives in time and $n$ continuous derivatives in space ($m,n\in\mathbb N_0$). Denote by $C(\R)$ the set of continuous functions  $f\colon \R\to \R$, and further by $C_b([0,T]\times\R)$ the set of bounded continuous functions. 

Denote by $L^1( \mathcal{S})$ the space of measurable functions $f$ such that 
\begin{align}
     \|f\|_1 = \int_{ \mathcal{S}} |f(s)| \, ds < \infty.
\end{align} 
Define  $W^{1,k}( \mathcal{S})$ as the space of functions in $L^1(\mathcal{S})$ whose weak derivatives up to order $k$ also belong to $L^1(\mathcal{S})$. For functions $f$, with domain $[0,T]\times \R$, we say that $f\in W^{1,0}([0,T])\cap L^1(\R)$ if $f_x(t)$ belongs to $W^{1,0}([0,T])$ for any $x\in \R$, and $f_t(x)\in L^1(\R)$. Also denote $\|f\|_{\infty}=\sup_{x\in \mathcal{S}} |f(x)|$ to be the supremum norm. 
In this paper, we use two different product notation. The first is between a measure $\mu$ on $\mathcal{S}$ and a function $G \colon \mathcal{S} \to \R$
\begin{align*}
    \langle \mu, G \rangle = \int_{\mathcal{S}} G(x) \, \mu(dx).
\end{align*}
The second product is between two functions $\rho, G \colon \mathcal{S} \to \R$, defined as
\begin{align*}
    \langle \rho, G \rangle = \int_\mathcal{S} G(x) \rho(x) \, dx.
\end{align*}

In order to describe the trajectory of measures on the space $\mathcal{M}_+^1$, we introduce the  the càdlàg space $D([0,T], \mathcal{M}_+^1(\mathcal{S}))$ for any $T > 0$, the space of functions $F : [0,T] \to \mathcal{M}_+^1(\mathcal{S})$ that are right-continuous with left limits. 

Whenever $N$ and $T$ are fixed, we denote by $\mathbb{P}$ and $\mathbb{E}$ the probability measure and expectation induced by the empirical Markov process $\pi_t^N$  on the càdlàg space $D([0,T], \mathcal{M}_+^1(\mathcal{S}))$, where $\mathcal{S}=\mathcal{M}_+^1([-M,M])$. In the càdlàg space, a trajectory $\{\mu_t^N \colon 0\leq t\leq T\}$ converges to the  trajectory $\{\mu_t \colon 0\leq t\leq T\}$ weakly if for every function $G_t\in C_b([0,T]\times\R)$ we have that
\begin{align*}
    \int_0^T \Hprod{\mu_t^N,G_t} dt \to \int_0^T \Hprod{\mu_t,G_t} dt 
\end{align*} in the weak topology of $\mathcal{M}_+^1(\mathcal{S})$.

\section{Proof of Theorem \ref{thm:HydrodinamiclimitClick}} \label{Sec:Proof2} We present the proof of Theorem \ref{thm:HydrodinamiclimitClick} in four steps. 

The first step is to use the entropy method (see  \cite{Kipnis1999}) to find a differential measure equation that the limit of the empirical measure $\pi_t^N$ must satisfy. Since $\pi_t^N$ is a distribution over $\mathcal{M}_+^1([-M,M])$, its limit $\pi_t$ also belongs to $\mathcal{M}_+^1(\mathcal{M}_+^1([-M,M]))$, and the first step only shows that such limit trajectories $\{\pi_t: 0\leq t\leq T\}$ exists. 

The second step is establishing uniqueness of the differential measure equation where, by Gr\"{o}nwall's inequality, it is possible to show that $\pi_t$ is not only an unique element of $\mathcal{M}_+^1(\mathcal{M}_+^1([-M,M]))$ for all $t\in[0,T]$, but consists in a concentrated distribution, i.e., a single element of $\mathcal{M}_+^1([-M,M])$. Therefore, the solution is an element of $D([0,T],\mathcal{M}_+^1([-M,M]))$. 

In order to better understand the nature of these trajectories, we demonstrate in the third step that if the initial distribution is uniform, then the trajectory remains uniform over time. Finally, in the last step, we deduce the differential equation \eqref{eq:convolution-differential-equation} for the density from the differential measure equation.

\subsection*{The measure equation} Following the arguments of Landim and Kipnis \cite{Kipnis1999}), by analyzing the Feller-Dynkin process, one obtains a non-trivial equation for measures that the random measure $\pi_t^N$, which is an element of $\mathcal{M}_+^1(\mathcal{M}_+^1([-M,M]))$, must satisfy in the limit.

For that,  fix $T>0$ and take any function $G_t\in C^{1,0}([0,T],\R)$. Then, we have that  
\begin{align*}
    \mathcal{L}_{N} \Hprod{\pi_t^N,G_t}&= \frac{1}{N} \sum_{x,y\in[N]} \frac{1}{N}  \left(2G_t\left(\frac{\omega_t(x)+\omega_t(y)}{2}\right)-G_t(\omega_t(x))-G_t(\omega_t(y))\right) \\
    &= 2\Hprod{\pi_t^N \ast  \pi_t^N,G_t(\cdot /2)}-2\Hprod{\pi_t^N,G_t}.
\end{align*} 
In this equation, for two finite measures $\mu$ and $\nu$ in the real line, for any measurable set $E\subset \R$, we define the convolution measure by
\begin{align*}
    (\mu*\nu)(E)=\int\int \ind_E(x+y) d\mu(x)d\nu(y),
\end{align*} so that
\begin{align}\label{eq:ConvolutionNformula}
    \pi_t^N \ast  \pi_t^N (du)= \frac{1}{N^2} \sum_{x,y\in [N]} \delta_{\omega_t(x)+\omega_t(y)}(du). 
\end{align}
Using the formula of the Dynkin martingale, see Appendix A1.5 of \cite{Kipnis1999}, we get that the process $M_t^{G,N}$, defined through 
\begin{align}\label{eq:Trajectorymartingalcompletegraph}
    \Hprod{\pi_t^N,G_t}= \Hprod{\pi_0^N,G_0}+ \int_0^t 2\left(\Hprod{\pi_s^N \ast  \pi_s^N,G_s(\cdot/2)}-\Hprod{\pi_s^N,G_s}\right)+\Hprod{\pi_s^N,\partial_s G_s}  ds + M_t^{G,N},
\end{align}
is a martingale. To control the second moment of the martingale, define the process
\begin{align*}
\mathcal{B}_t^{G,N} & =\mathcal{L}_{N,\mu}^d \left(\Hprod{\pi_t^N,G_t}\right)^2 -2\Hprod{\pi_t^N,G_t}\mathcal{L}_{N,\mu}^d \Hprod{\pi_t^N,G_t}\\
&= \frac{1}{N^3} \sum_{(x,y) \in [N]} \left(2G_t\left(\frac{\omega_t(x)+\omega_t(y)}{2}\right)-G_t(\omega_t(x))-G_t(\omega_t(y))\right)^2.
\end{align*} 
By the Carré du Champ formula (see Appendix A1.5 of \cite{Kipnis1999}), we obtain that
\begin{align*}
  \Big(M_t^{G,N}\Big)^2 -  \int_0^t\mathcal{B}_s^{G,N} ds
\end{align*}
is a  (zero mean) martingale. Observe that the function $G_t$ is bounded for every $\omega_t(x) \in [-M, M]$. 
Moreover, since the (double) sum has only $N^2$ summands, we conclude that $\mathcal{B}_t^{N,G}$ converges to zero as $N \to \infty$, uniformly for all $t \in [0,T]$. Thus, by Doob's inequality,  the supremum of  $M_t^{G,N}$ on any time interval $[0,T]$ converges to zero in probability.

To prove that the trajectory $\pi_t^N$ is tight, we apply Prokhorov's Theorem (see, e.g.\ Theorem 1.3 in the Chapter 4 of \cite{Kipnis1999}) to a push forward argument that requires that the opinions are bounded in $[-M,M]$ (see Proposition 1.7 in Chapter 4 of \cite{Kipnis1999}). To this end, we need to show that $\left\{ \Hprod{\pi^N_t , G_t} \colon 0\leq t\leq T \right\}$ is tight in $ D([0,T], \mathcal{M}(\mathbb{R})) $.  

Observe that $\omega_t(x)\in[-M,M]$ for every $x\in [N]$. Therefore, defining the constant $C_0(G)=\max\{G_t(x)\colon \,t\in[0,T],x\in[-M,M]\}$, we get that
\begin{align*}
    \P{\sup_
{0\leq t\leq T}\left|\Hprod{\pi_t^N,G_t}\right|\leq C_0(G)}=1 \ \text{ for every } N>0.
\end{align*}

To control the increment of the process between the times $\tau$ and $\tau + \theta$ (rather $\min\{\tau + \theta, T\}$), note that
\begin{align*}
    \Hprod{\hatpi_{\tau + \theta}^N, G_{\tau+\theta}} - \Hprod{\hatpi_{\tau}^N, G_{\tau}} 
    = & \, M_{\tau + \theta}^{G, N} - M_{\tau}^{G, N} \\
    & + \int_{\tau}^{\tau + \theta} 2\left(\Hprod{\pi_s^N \ast \pi_s^N, G_s(\cdot / 2)} - \Hprod{\pi_s^N, G_s}\right) + \Hprod{\pi_s^N, \partial_s G_s} \, ds.
\end{align*}

Thus, again using the maximum of $G_t$, for any $\tau, \tau + \theta \in [0, T]$, by taking the differences in equation \eqref{eq:Trajectorymartingalcompletegraph}, we have for every $N > 0$ that
\begin{align}\label{eq:secondmomentMincrement}
    \mathbb{E}\big[(M_{\tau + \theta}^{G, N} - M_{\tau}^{G, N})^2\big] 
    = \mathbb{E}\left[\int_{\tau}^{\tau + \theta} \mathcal{B}_t^{G, N} \, dt\right] 
    \leq \frac{C_1(G) \theta}{N},
\end{align}
and since $\partial_s G_s$ is also continuous, it attains a maximum on the set $[0, T] \times [-M, M]$. Furthermore, by \eqref{eq:ConvolutionNformula}, the convoluted measure can be bounded by the maximum of $G$ on the set $[0, T] \times [-2M, 2M]$. In particular, we obtain
\begin{align*}
    \int_{\tau}^{\tau + \theta} 2\left(\Hprod{\pi_s^N \ast \pi_s^N, G_s(\cdot / 2)} - \Hprod{\pi_s^N, G_s}\right) 
    + \Hprod{\pi_s^N, \partial_s G_s} \, ds \leq C_2(G) \theta,
\end{align*}
where $C_1(G)$ and $C_2(G)$ are constants depending solely on $G_t(x)$.

Therefore, using the second moment \eqref{eq:secondmomentMincrement} and applying Chebyshev's inequality, one can show that the increment of the process $\Hprod{\pi_t^N, G_t}$ in the space $D([0,T], \mathcal{M}(\R))$ is small with probability converging to one. This establishes that the trajectories $\Hprod{\pi_t^N, G_t}$ are tight.

Now, let $Q^N$ be the measure induced by the trajectory $\pi_t^N$ in $D([0,T], \mathcal{M}_+^1(\mathcal{M}_+^1([-M,M])))$. By Proposition 1.7 in \cite{Kipnis1999}, since $G_t\in C^{1,0}([0,T],\R)$, the tightness of the trajectories $\Hprod{\pi_t^N,G_t}$ over $D([0,T],\mathcal{M}(\R))$ imply tightness of the measure sequence $Q^N$ by a push forward argument.  

With tightness established in the càdlàg space, any possible limit distribution of $Q^N$ is an element of  $D([0,T],\mathcal{M}_+^1(\mathcal{S}))$ , where $\mathcal{S}=\mathcal{M}_+^1([-M,M])$. Now, let us prove that the limit measure is concentrated on trajectories whose measures are solutions of the differential measure equation
\begin{align}\label{eq:measureequation}
    \Hprod{\pi_t,G_t}= \Hprod{\nu,G_0}+ \int_0^t 2\left(\Hprod{\pi_s \ast  \pi_s,G_s(\cdot/2)}-\Hprod{\pi_s,G_s}\right)+\Hprod{\pi_s,\partial_s G_s}  ds, \quad \forall t\in[0,T],
\end{align} where $G_t\in C^{1,0}([0,T],\R)$.

Let $Q^*$ be a limit point of the subsequence $Q^{N_k}$. Since $G_t \in C^{1,0}([0,T],\mathbb{R})$, by weak convergence, for every $\varepsilon > 0$ and any trajectory $\{\pi_t, 0 \leq t \leq T\}$, we obtain
\begin{align*}
    \liminf_{k \to \infty}&\, Q^{N_k} \left( \sup_{0 \leq t \leq T} \left| \begin{aligned} \langle \pi_t, G_t \rangle &- \langle \pi_0, G_0 \rangle  -\int_0^t 2\left(\Hprod{\pi_s \ast  \pi_s,G_s(\cdot/2)}-\Hprod{\pi_s,G_s} \right)ds\\&-\int_0^t\Hprod{\pi_s,\partial_s G_s}  \, ds \end{aligned} \right| > \varepsilon \right)\\
\geq &\,Q^* \left( \sup_{0 \leq t \leq T} \left| \begin{aligned} \langle \pi_t, G_t \rangle &- \langle \pi_0, G_0 \rangle  -\int_0^t 2\left(\Hprod{\pi_s \ast  \pi_s,G_s(\cdot/2)}-\Hprod{\pi_s,G_s} \right)ds\\&-\int_0^t\Hprod{\pi_s,\partial_s G_s}  \, ds \end{aligned} \right| > \varepsilon \right),
\end{align*} since the convolution operator $\ast$ is continuous in the space of trajectories, and for a fixed $G_t \in C^{1,0}([0,T],\mathbb{R})$, the supremum is continuous with respect to the trajectory. 

Considering the definition of the configuration at time zero and applying the strong law of large numbers, it is simple to show that
\begin{align*}
    \lim_{N \to \infty} \Hprod{\pi_0^N, G_0} = \Hprod{\nu,G_0}.
\end{align*}Moreover, if $\pi_t^N$ converges weakly to  $\pi_t$, we have that $\pi_t^N\ast \pi_t^N$ also converges weakly to $\pi_t\ast \pi_t$.  Then, applying Chebyshev's and Doob's inequalities, one gets
\begin{align} \label{eq:DoobClick}
 Q^N \left[ \sup_{0 \leq t \leq T} |M_t^{G,N}| \geq \epsilon \right] \leq 4 \epsilon^{-2} \mathbb{E}_{Q^N} \left[ \int_0^T \mathcal{B}_s^{G,N} \, ds \right] \leq \frac{C(G) T}{\epsilon^2 N}.
\end{align} This implies that every limiting measure $\pi_t$ must satisfy the differential measure equation \eqref{eq:measureequation}.

\subsection*{Uniqueness of the measure equation}For each fixed time $t \in [0,T]$, the existence of at least one solution to equation \eqref{eq:measureequation} can be guaranteed by the tightness of the trajectories $(\pi_t^N)_N$. Uniqueness, on the other hand, is a separate issue.

\begin{proposition}\label{Prop:Uniqueness}
    Consider any measure $\nu \in \mathcal{M}_+^1([-M,M])$. Then there exists a unique trajectory of measures $\{\pi_t: 0\leq t\leq T\}$ in $D([0,T],\mathcal{M}_+^1([-M,M]))$ that satisfies
    \begin{align}\label{Eq:equaçãodemedida}
        \begin{cases}
            \Hprod{\pi_t,G_t}= \Hprod{\nu,G_0}+ \int_0^t 2\left(\Hprod{\pi_s \ast  \pi_s,G_s(\cdot/2)}-\Hprod{\pi_s,G_s}\right)+\Hprod{\pi_s,\partial_s G_s}  ds,\,
            \forall t\in[0,T];\\
        \pi_0=\nu, 
        \end{cases}
    \end{align}
    for all $G_t\in C^{1,0}([0,T],\R)$.
\end{proposition}

\begin{proof}
    Existence is clarified in the previous subsection, so we only have to deal with uniqueness. The proof exploits Gr\"{o}nwall's inequality. For any distribution $\pi\in\mathcal{M}_{+}^1 \left( \mathcal{M}_+^1([-M,M]) \right)$, we introduce the norm 
    \begin{align*}
        \|\pi\|=\sup\left\{\sqrt{\Var{\Hprod{\pi,G}}}: G\in C(\R)\text{  s.t. } \|G\|_{\infty}=1\right\}.
    \end{align*}
    Since, the standard deviation is a norm in the space of random variables, it is easy to see that $\|.\|$ is a norm over $\mathcal{M}_+^1(\mathcal{M}^1_+([-M,M]))$. 
    
   Suppose that $\{\pi_t^1: 0\leq t\leq T\}$ and $\{\pi_t^2: 0\leq t\leq T\}$ are two solutions of equation \eqref{Eq:equaçãodemedida}. Thus, for every $t\in[0,T]$, one gets for every $G_t\in C^{1,0}([0,T],\R)$ that
    \begin{align}\label{eq:equaçãonadiferença}
        \Hprod{\pi_t^1-\pi_t^2,G_t}= 2\int_0^t \Hprod{(\pi_s^1+\pi_s^2)\ast(\pi_s^1-\pi_s^2),G_s(\cdot/2)}-\Hprod{\pi_s^1-\pi_s^2,G_s-\partial G_s}ds.
    \end{align} 
    
    By applying the standard deviation norm in the equation \eqref{eq:equaçãonadiferença}, 
    we get
    \begin{align*}
        \sqrt{\Var{\Hprod{\pi_t^1-\pi_t^2,G}}}\leq  2\int_0^t \sqrt{\Var{\Hprod{(\pi_s^1+\pi_s^2)\ast(\pi_s^1-\pi_s^2),G(\cdot/2)}}}+\sqrt{\Var{\Hprod{\pi_s^1-\pi_s^2,G}}} ds.
    \end{align*} Since the trajectories are distributions over $\mathcal{M}_+^1([-M,M])$, we have $(\pi_s^{1}+\pi_s^{2})(\mathbb{R}) \leq 2$. 
    Thus, defining
\begin{align*}
    F_s(y) = \int_{\R} G\left(\frac{x+y}{2}\right) (\pi_s^1+\pi_s^2)(dx) \leq 2\|G\|_{\infty},
\end{align*}  
and maximizing over functions $G \in C(\mathbb{R})$ such that $\|G\|_{\infty} = 1$, we obtain  
\begin{align*}
    \|\pi_t^1 - \pi_t^2\| \leq 6 \int_0^t \|\pi_s^1 - \pi_s^2\| ds.
\end{align*}  
Applying Gr\"{o}nwall's inequality (see Appendix B, Section k of \cite{Evans2010}), this implies that  
\begin{align*}
    \|\pi_t^1 - \pi_t^2\| \leq \|\pi_0^1 - \pi_0^2\| e^{6t}.
\end{align*}  
Since both distributions are initially equal to $\nu$ at time zero, we conclude that $\pi_t^1=\pi_t^2$ for \emph{fixed} $t\in[0,T]$.

Now, consider the trajectory $\{\pi_t, 0 \leq t \leq T\}$ that satisfies equation \eqref{Eq:equaçãodemedida}, and let $\pi_t^1$ and $\pi_t^2$ be two independent samples from the same distribution. By the same argument, we obtain that  
$\Var{\Hprod{\pi_t^1 - \pi_t^2, G}} = 0$ for all $G \in C^0(\mathbb{R})$, which implies that $\{\pi_t\colon 0 \leq t \leq T\}$ is not a trajectory of distributions in $\mathcal{M}_+^1(\mathcal{M}_+^1([-M,M]))$, but rather a trajectory of measures belonging to $\mathcal{M}_+^1([-M,M])$, completing the proof.
\end{proof}

\subsection*{Absolutely continuous limit} The measure equation \eqref{eq:measureequation} is a powerful tool for understanding the averaging process. However, solving and analysing a differential measure equation is not a trivial task. Therefore, in this Section, assuming the existence of a measure trajectory $\{\pi_t: 0 \leq t \leq T\}$, we will use a probabilistic argument to show that if $\nu$ is absolutely continuous, then $\pi_t$ is also absolutely continuous. In particular, in this case, the density of $\pi_t$ will satisfy a differential equation \eqref{eq:convolution-differential-equation}.

For a fixed value of $t > 0$, we aim to establish that the limit of $\pi_t^N$ is absolutely continuous with respect to the Lebesgue measure. The proof relies on a quantitative argument that estimates the number of opinions that, at time $t$, end up in any given small interval on the real line. Our goal is to show that this number is small and thus show that $\pi_t$ is diffusive.

To achieve this, we divide the initial opinions into two groups: an arbitrary small group of vertices that changes its opinion a lot of times until $t$, and a second group which represents the majority of the vertices and consists of those that do not change their opinions significantly many times. We prove that this second group cannot have a sharp concentration, since its  opinions do not move enough. Moreover, since the group of vertices that frequently change opinions can be made arbitrarily small, no set will have a large concentration of final opinions at time $t$, implying that the limiting measure is absolutely continuous.

The first step in this proof is to control how much an opinion changes up to time $t$. To this end, define $\xi_t(x)$ as the number of times the opinion of $x$ has attempted to change up to time $t$, that is, the number of times $x$ interacts with another vertex. Then, consider
    \begin{align}
        \label{eq:X_j^t} X_j^t&=\sum_{x\in [N]} \ind\{\xi_t(x)=j\},
    \end{align} 
    which counts the number of vertices whose opinion changes exactly $j$ times up to time $t$.

    \begin{lemma} \label{Lem:X_jconcentration}
        There exists $N_0 = N_0(t)$ such that for every $N > N_0$ and every $j > 0$, one can find a constant $C = C(j,t)$ such that, for every $\alpha > 0$, we have
\begin{align*}
    \mathbb{P} \left( \left| X_j^t - N \frac{(2t)^{j}}{j!} \left(1-\frac{1}{2N}\right)^{j}e^{-2t\left(1-\frac{1}{2N}\right)} \right| \geq \alpha \sqrt{N} \right) \leq \frac{C}{\alpha^2}.
\end{align*} 
    \end{lemma}
\begin{proof} 
The proof uses a second moment argument. 
Since, any vertex $x$ belongs to $2N-1$ pairs, at time $t$, the random variable $\xi_t(x)$ is a Poisson random variable with parameter equal to $2t\left(1-\frac{1}{2N}\right)$. 
Specifically,
    \begin{align*}
        \Ex{X_j^t}&=N \P{\xi_t(1)=j}= N \frac{(2t)^j}{j!} \left(1-\frac{1}{2N}\right)^j e^{-2t\left(1-\frac{1}{2N}\right)}.
    \end{align*}
    In the averaging process, the opinion of each pair $\{x,y\}$ changes according to an exponential random variable with rate $2/N$. Let $\Xi_{\{x,y\}}^t$ denote the number of times the pair $\{x,y\}$ attempts to change its state up to time $t$. By this observation, this random variable follows a Poisson distribution with rate $2t/N$. Then, we have
    \begin{align*}
    \Ex{(X_j^t)^2}&=\sum_{x,y\in [N]} \P{\xi_t(x)=j, \xi_t(y)=j}
    =\Ex{X_j^t} + N(N-1)\P{\xi_t(1)=j,\xi_t(2)=j}\\
    &=\Ex{X_j^t} + N(N-1)\sum_{\ell=0}^j \P{\xi_t(1)=j,\xi_t(2)=j, \Xi_{\{1,2\}}^t=\ell}.
    \end{align*}
    Removing the edge $\{1,2\}$, the variables $\xi_t(1)$ and $\xi_t(2)$ are two independent Poisson random variables with parameter $2t\left(1-\frac{3}{2N}\right)$. This yields
    \begin{align*}
    \Ex{(X_j^t)^2}&\leq \Ex{X_j^t} + N(N-1)\left( \left(\left(2t\left(1-\frac{3}{2N}\right)\right)^{j}\frac{e^{-2t\left(1-\frac{3}{2N}\right)}}{j!}\right)^2 e^{-\frac{2t}{N}} + \P{\Xi_{\{1,2\}}^t>0}\right)\\
    &=\Ex{X_j^t} + N(N-1)\left( \left(\left(2t\left(1-\frac{3}{2N}\right)\right)^{j}\frac{e^{-2t\left(1-\frac{3}{2N}\right)}}{j!}\right)^2 e^{-\frac{2t}{N}} +(1-e^{-\frac{2t}{N}})\right).
\end{align*}
With $t$ fixed, there exists $N_0 = N_0(t)$ (a value that does not depend on the number of movements $j$) such that, for every $N > N_0$ and any $j \geq 0$, we have
\begin{align*}
    N(N-1)(1-e^{-\frac{2t}{N}}) &\leq 4t(N-1), \\
    N^2 \left(\left(2t\left(1-\frac{3}{2N}\right)\right)^{j}\frac{e^{-2t\left(1-\frac{3}{2N}\right)}}{j!}\right)^2 e^{-\frac{2t}{N}} 
    &- N^2 \frac{(2t)^{2j}}{(j!)^2}\left(1-\frac{1}{2N}\right)^{2j} e^{-4t\left(1-\frac{1}{2N}\right)} \leq 2t N \frac{(2t)^{2j}}{(j!)^2} e^{-4t}.
\end{align*}
In particular, for every $j$ and $N > N_0$, we find that $\mathbb{E}[X_j^t]$ and $\mathrm{Var}(X_j^t)$ are of order $N$. Then, the claim follows by Chebyshev's inequality.  
\end{proof}

Note that the mass of the measure $\pi_t^N$ is concentrated at the points $\omega_t$, while in equation \eqref{eq:EmpiricalmeasureinTd}, the measure $\hatpi_t^N$ represents a weighted distribution on the torus. This seemingly small difference has significant consequences for the hydrodynamic limit equation. For instance, the solution to the heat equation smooths discontinuities, while the measure $\pi_t^N$  in the complete graph can amplify them even further. In particular, by Lemma \ref{Lem:X_jconcentration}, when starting from an initial configuration where the opinions are independently distributed on $\{0,1\}$ according to a Bernoulli distribution, the limiting distribution of $\pi_t^N$ assigns positive probability to all dyadic numbers, resulting in a measure that is not absolutely continuous in the limit; see Figure \ref{fig:SimulationBer}.

For this specific reason, we require the extra conditions of the density $\nu$, that is $\nu$ is absolutely continuous with respect the Lebesgue measure $\lambda$. 

\begin{proposition}\label{Prop:Abscontinuous}
    Let $\nu\ll \lambda$ have compact support $[-M,M]$. Then for every fixed $t>0$, the limit measure $\pi_t$ of the sequence $(\pi_t^N)_N$ is also absolutely continuous w.r.t.\ Lebesgue measure. 
\end{proposition}
\begin{proof}
 By Proposition \ref{Prop:Uniqueness}, for every $t\in [0,T]$ there exists a unique limit measure $\pi_t$ for the random measures $\pi_t^N$, our goal is to show that $\pi_t$ is absolutely continuous with respect to Lebesgue measure. Consider any nested union of open intervals $ (\mathcal{I}_k)_k \subset [-M, M] $ such that
     \begin{align}\label{eq:Limitlebesguemeasuretozero}
         \lim_{k\to \infty} \lambda(\mathcal{I}_k)=0.
     \end{align} We aim to prove that $\lim_{k \to \infty} \pi_t(\mathcal{I}_k) = 0$. Consequently, since the set of intervals is nested, it follows that for every set $A$ with Lebesgue measure zero, we also get $\pi_t(A) = 0$ almost surely, concluding the proof. 
    
    We now proceed with an indirect argument and assume that there exists a nested union of open intervals $(\mathcal{I}_k)_k$ satisfying condition \eqref{eq:Limitlebesguemeasuretozero}, but the measure $\pi_t$ converges to a positive value, that is,
    \begin{align}\label{eq:Limitmeasureintervalpositive}
    \lim_{k\to \infty} 
 \pi_t(\mathcal{I}_k)>\eta>0.
    \end{align} In other words, if \eqref{eq:Limitmeasureintervalpositive} happens, fixing any set $\mathcal{I}$ of the nested sequence $(\mathcal{I}_k)_k$, when $N$ is large, the random measure $\pi_t^N$ with high probability places more than $\eta N$ opinions in $\mathcal{I}$.  To demonstrate that this is not the case, we need a quantitative argument that provide bounds on the number of opinions in $\mathcal{K}_N$ that belong to sets $\mathcal{I}$ with small Lebesgue measure at time $t$. 

The quantitative argument consists of two main steps. First, we divide the set of opinions into two categories: those that change more than $j$ times and those that do not. The first step of the proof addresses the points that change more than $j$ times, while the second step calculates the probability of points that do not change as frequently. Specifically, consider
    \begin{align}
    X_{>j}^t&=\sum_{x\in [N]} \ind\{\xi_t(x)>j\},\nonumber
    \end{align} the total number of opinions that that changes more than $j$ times. As a key observation, for every fixed $j$, by Lemma \ref{Lem:X_jconcentration}, since the values of $\{X_0^t, \ldots, X_j^t\}$ are concentrated for every $N>N_0(t)$, the remainder of the $N$ particles that move more than $j$ times also exhibit concentration. Therefore, $X_{>j}^t$ fluctuates around its mean. In particular, one can find $j=j(\eta)$, such that $\P{\mathrm{Poi}(t)>j}<\frac{\eta}{2}$, and then 
    \begin{align}\label{eq:limtX_>j}
            \lim_{N \to \infty} \P{X_{>j}^t < \frac{\eta}{2} N} = 1.
    \end{align} In our construction, assume that whenever an opinion moves more than $j$ times, it will end up in the final set $\mathcal{I}$ thereby increasing the value $\pi(\mathcal{I})$ by $\frac{\eta}{2}$. 

To handle the points that do not move as much, we now consider the value of $\lambda(\mathcal{I})$ and the fact that the final opinion of a vertex, when it changes only a finite number of times, is sensible to initial perturbations. The proof proceeds by establishing a bound on the probability that a particle ends up in $\mathcal{I}$. This bound relies on associating each set $\mathcal{I}$ and each particle $x \in [N]$ with a random set $\widehat{\mathcal{I}}$, such that if $\omega_0(x) \in \widehat{\mathcal{I}}$, then $\omega_t(x) \in \mathcal{I}$. The probability bound essentially depends on the fact that $\lambda(\widehat{\mathcal{I}})$ is small and that $\nu$ is absolutely continuous.

  Let $\delta \in \R$ and fix any $x \in [N]$. Define an initial perturbed configuration at $x$ as $(\widehat{\omega}_0^{\delta}(y))_{y\in [N]}$, where $\widehat{\omega}_0^\delta(x) = \omega_0(x) + \delta$ and $\widehat{\omega}_0^\delta(y) = \omega_0(y)$ for every $y \in [N] \setminus \{x\}$. By counting the number of times that the opinion $\omega_t(x)$ changes, for instance $j$ times, we have the following bounds
    \begin{align}\label{eq:pertubedboundmovement0}
        \begin{cases}
            \widehat{\omega}_t^\delta(x)\geq \omega_t(x) + \frac{\delta}{2^j}, &\text{ if } \delta>0;\\
            \widehat{\omega}_t^\delta(x)\leq \omega_t(x) + \frac{\delta}{2^j}, &\text{ if } \delta<0. 
        \end{cases}
    \end{align} In particular, for every $\e> 0$ fixed and interval $(a,b)$ with measure $\e$, if the opinion of the vertex $x$ changes less than $j = -\log_2\left(\frac{\e}{\delta}\right)$ times, then perturbations of the initial opinion can move the opinion $\omega_t(x)$ a distance greater than $\e>0$, and thus possibly outside the interval $(a,b)$. With this argument, one can, at time $t$, relate the probability of opinion $\omega_t(x)$ with the probability at time zero, where the opinions are independent. 
    
    Fixing $t$, consider an arbitrary vertex $x$ such that $\{\xi_t(x)\leq j\}$ happens, and examine all its possible outcomes after an initial perturbation of size $c\in(-\delta,\delta)$. Any possible perturbed initial opinion $\widehat{\omega}_0^c$ of that vertex $x$, belongs to the interval $B_{\omega_0(x)}(\delta)$, and, using \eqref{eq:pertubedboundmovement0}, $\widehat{\omega}_t^c$ belongs after time $t$ to the set
    \begin{align*}
         \left\{\widehat{\omega}_t^c(x): c\in(-b,b)\right\} \supseteq \left( \omega_t(x) - \frac{\delta}{2^j}, \omega_t(x) + \frac{\delta}{2^j} \right).
    \end{align*} 
    Since the process is Feller, small changes in the initial distribution result in small changes in the outcome. This ensures that intervals are mapped to intervals, in the map $\delta\mapsto \widehat{\omega}^\delta_t$ . For more on this property, see the discussion in \cite{Gantert2020}.

\begin{figure}[ht!]
    \centering
    \includegraphics[scale=.7]{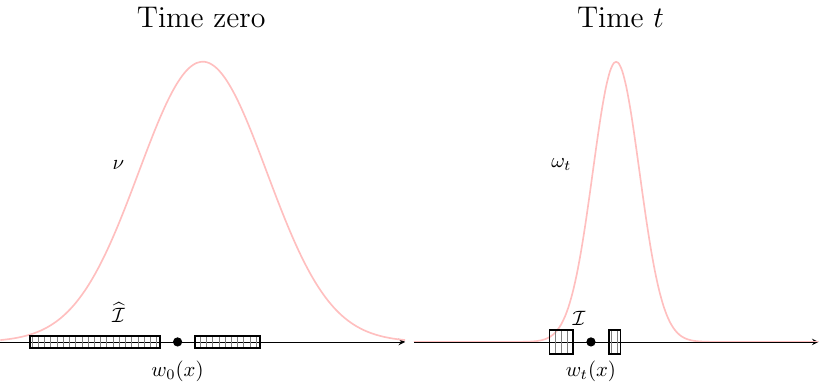}
   \caption{Side by side, the initial distribution of opinions at time zero and at time $ t $. The set $ \mathcal{I} $ at time $ t $ is related to the set $ \widehat{\mathcal{I}} $ at time zero. In this case, intervals are mapped to intervals, and using perturbation estimates, one can derive bounds for the Lebesgue measure of the set $\widehat{\mathcal{I}}$.
}
    \label{fig:Initialandendsets}
\end{figure}

    To compute the contribution of the set of points that change less than $j$ times to the final value of $\pi(\mathcal{I})$, we shift the problem at time $t$ to a problem at time zero, where the opinions are independent. First, for any set $\mathcal{I}$ and vertex $x$, consider the random set $\widehat{\mathcal{I}}= \widehat{\mathcal{I}}(\omega,\mathcal{I})=\{\omega_0(x): \omega_t(x)\in \mathcal{I}\}$ of possible starting values when we require to be in $\mathcal I$ after time $t$.
    
    Since $(\mathcal{I}_k)_k$ is a nested union of open interval, without loss of generality one can find a disjoint union of open intervals where $\mathcal{I}=\bigcup_{i=1}^\infty (a_i,b_i)$. Then, for each interval of the union, for instance $(a_1,b_1)$ with total measure $|b_1-a_1|=\e_1<\e=\lambda(\mathcal{I})$, one can use the perturbation argument, to find an interval $(\widehat{a_1},\widehat{b_1})$ at time zero, with total measure smaller than $2^j \e_1$, such that if $\omega_0(x)\in (\widehat{a_1},\widehat{b_1})$ then $\omega_t(x)\in (a_1,b_1)$.  By a simple construction, for each vertex $x\in [N]$ such that $\{\xi_t(x)<j\}$, moving the initial opinion around, one can show that $\lambda(\widehat{\mathcal{I}})\leq \e 2^j$. See in Figure \ref{fig:Initialandendsets}, a representation of the problem at time $t$ and at time zero.
     
    Observe that the initial opinion $\omega_0(x)$ is independent of the set $\widehat{\mathcal{I}}$. Since $\nu$ is absolutely continuous with respect to Lebesgue measure, for every $\eta > 0$, $\ell>0$, there exists an $\e = \e(\eta, \nu, j,\ell) > 0$ such that for every measurable set $A$ with $\lambda(A) < 2^j \e$, we have $\nu(A) < \frac{\eta}{2\ell}$. In particular, it is true that
    \begin{align*}
        \P{\omega_t(x)\in \mathcal{I}|\xi_t(x)<j}&= \P{\omega_0(x)\in \widehat{\mathcal{I}}|\xi_t(x)<j}<\frac{\eta}{2\ell}.
    \end{align*}

    Then, assuming all points that move more than $j$ times end up in the limiting set, by the limit of equation \eqref{eq:limtX_>j}, for any $\ell>0$, by Markov's inequality, for every set $\mathcal{I}$ in the sequence $(\mathcal{I}_k)_k$ such that $\lambda(\mathcal{I})<2^j\e$, we got that
    \begin{align*}
    \lim_{N\to \infty}\P{\pi_t^N(\mathcal{I})>\eta} &\leq \lim_{N\to \infty}  \P{\sum_{x\in [N]}  \ind\left\{\omega_t(x)\in \mathcal{I},\xi_t(x)<j\right\}\geq N \frac{\eta}{2} }\leq  \frac{1}{\ell}.
    \end{align*} 
    Now set $\ell>1/\eta$ to derive a contradiction with \eqref{eq:Limitmeasureintervalpositive}. The proposition follows. 
\end{proof}

\subsection*{The differential equation} 
In conclusion, starting from an absolutely continuous density, by Proposition~\ref{Prop:Abscontinuous} and the uniqueness of the trajectory of Proposition \ref{Prop:Uniqueness}, the measure equation \eqref{eq:measureequation} for the absolutely continuous limit trajectory $\{\pi_t=\rho(t,x)dx : 0\leq t\leq T\}$ satisfies for every $t\in[0,T]$ the equation  
\begin{align}\label{eq:Integral-conv-equation}
    \langle \rho_t, G_t \rangle = \langle \rho_0, G_0 \rangle + \int_0^t 2\left(\Hprod{\rho_s \ast  \rho_s(\cdot),G_s(\cdot/2)}-\Hprod{\rho_s,G_s}\right)+\Hprod{\rho_s,\partial_s G_s}  \, ds,
\end{align} where $G_t\in C^{1,0}([0,T],\R)$. 
Using the substitution $z=\frac{x}{2}$ obtains
\begin{align*}
     2\Hprod{\rho_s \ast  \rho_s(\cdot),G_s\left(\frac{\cdot}{2}\right)} &=  2
    \int_{\R} G_s\left(\frac{x}{2}\right)\rho_s \ast  \rho_s(x) dx  \\
    &= 4  \int_{\R} G(z)\, \rho_s \ast \rho_s(2z) dz = 4\Hprod{\rho_s\ast \rho_s( 2 \cdot ), G_s(\cdot)} . 
\end{align*}As a consequence of Proposition \ref{Prop:Abscontinuous}, the measure $Q^*$ is concentrated in paths $\{\pi_t=\rho(t,x)dx:0\leq t\leq T\}$ that satisfy
\begin{align*}
    \langle \rho_t, G_t \rangle &= \langle \rho_0, G_0 \rangle + \int_0^t 4\Hprod{\rho_s \ast  \rho_s(2\cdot),G_s}-2\Hprod{\rho_s,G_s}+\Hprod{\rho_s,\partial_s G_s}  \, ds,\\
    \int_0^t \partial_s\langle \rho_s, G_s \rangle ds &= \int_0^t 4\Hprod{\rho_s \ast  \rho_s(2\cdot),G_s}-2\Hprod{\rho_s,G_s}+\Hprod{\rho_s,\partial_s G_s}  \, ds.
\end{align*} 
Consequently, $\rho(s,x)$ has a weak time derivative, thus it belongs to $W^{1,1}([0,T])\cap L^1(\R)$ and satisfies the weak equation
\begin{align}\label{eq:weak-convoluted-equation}
    \int_0^t \Hprod{\partial_s \rho_s- 4\rho_s \ast  \rho_s(2\cdot)+2\rho_s,G_s}  \, ds=0, \quad\forall t\in[0,T],
\end{align} where $G_s$ can be any function in $C^{1,0}([0,T],\mathbb{R})$. Moreover, since $C^{1,0}([0,T],\mathbb{R})$ is dense in $L^1([0,T],\mathbb{R})$, one can choose a sequence of continuous functions converging to the function $\partial_s \rho_s - 4\rho_s \ast \rho_s(2\cdot) + 2\rho_s$, making the integral equal to the norm. As a consequence, the weak solution is a strong solution, completing the proof of Theorem \ref{thm:HydrodinamiclimitClick}.

\subsection*{Acknowledgments.} We are grateful to Dirk Bl\"omker, Milton Jara and Stefan Großkinsky for enlighting  discussions. A.M.C. is supported by  PROBRAL-CAPES through Project 88887.979525/2024-00. T.F. acknowledges: the present work has been  supported by Coordenação de Aperfeiçoamento de Pessoal de Nível Superior- Brasil (CAPES) under financial code 8888.700851/2022-01;  has been supported by the  National Council for Scientific and Technological Development (CNPq) via a Universal Grant Number 406001/2021-9 and a Bolsa de Produtividade number 311894/2021-6; and has been supported by FAPESB (Edital FAPESB Nº 012/2022 - Universal - NºAPP0044/2023).  M.H. and M.S. are supported by Deutscher Akademischer Austauschdienst (DAAD) and Bundesministerium für Bildung und Forschung (BMBF). 

\bibliographystyle{plain}
\bibliography{ref}

\begin{thebibliography}{10}

\bibitem{Aldous2012}
David Aldous and Daniel Lanoue.
\newblock A lecture on the averaging process.
\newblock {\em Probability Surveys}, 9(none), January 2012.

\bibitem{https://doi.org/10.48550/arxiv.2401.06668}
Luisa Andreis, Wolfgang K\"{o}nig, Heide Langhammer, and Robert I.~A. Patterson.
\newblock Spatial particle processes with coagulation: Gibbs-measure approach, gelation and smoluchowski equation, 2024.

\bibitem{Chatterjee2022}
Sourav Chatterjee, Persi Diaconis, Allan Sly, and Lingfu Zhang.
\newblock A phase transition for repeated averages.
\newblock {\em The Annals of Probability}, 50(1), January 2022.

\bibitem{Cox1986}
J.~Theodore Cox and David Griffeath.
\newblock Diffusive clustering in the two dimensional voter model.
\newblock {\em The Annals of Probability}, 14(2), April 1986.

\bibitem{Deffuant2000}
Guillaume Deffuant, David Neau, Frederic Amblard, and Gérard Weisbuch.
\newblock Mixing beliefs among interacting agents.
\newblock {\em Advances in Complex Systems}, 03(01n04):87–98, January 2000.

\bibitem{Evans2010}
Lawrence Evans.
\newblock {\em Partial Differential Equations}.
\newblock American Mathematical Society, March 2010.

\bibitem{Gantert2020}
Nina Gantert, Markus Heydenreich, and Timo Hirscher.
\newblock Strictly weak consensus in the uniform compass model on $\mathbb{Z}$.
\newblock {\em Bernoulli}, 26(2), May 2020.

\bibitem{Nina2024}
Nina Gantert and Timo Vilkas.
\newblock The averaging process on infinite graphs, 2024.

\bibitem{Hendriks1983}
E.~M. Hendriks, M.~H. Ernst, and R.~M. Ziff.
\newblock Coagulation equations with gelation.
\newblock {\em Journal of Statistical Physics}, 31(3):519–563, June 1983.

\bibitem{Kipnis1999}
Claude Kipnis and Claudio Landim.
\newblock {\em Scaling Limits of Interacting Particle Systems}.
\newblock Springer Berlin Heidelberg, 1999.

\bibitem{kipnis1982heat}
Claude Kipnis, Carlo Marchioro, and Errico Presutti.
\newblock Heat flow in an exactly solvable model.
\newblock {\em Journal of Statistical Physics}, 27:65--74, 1982.

\bibitem{McLeod1962}
J.~B. McLeod.
\newblock On an infinite set of non-linear differential equations.
\newblock {\em The Quarterly Journal of Mathematics}, 13(1):119–128, 1962.

\bibitem{Melzak1957}
Z.~A. Melzak.
\newblock A scalar transport equation.
\newblock {\em Transactions of the American Mathematical Society}, 85(2):547–560, 1957.

\bibitem{Norris1999}
James~R. Norris.
\newblock Smoluchowski’s coagulation equation: uniqueness, nonuniqueness and a hydrodynamic limit for the stochastic coalescent.
\newblock {\em The Annals of Applied Probability}, 9(1), February 1999.

\bibitem{Quattropani2023}
Matteo Quattropani and Federico Sau.
\newblock Mixing of the averaging process and its discrete dual on finite-dimensional geometries.
\newblock {\em The Annals of Applied Probability}, 33(2), April 2023.

\bibitem{Risken1996}
Hannes Risken.
\newblock {\em The Fokker-Planck Equation: Methods of Solution and Applications}.
\newblock Springer Berlin Heidelberg, 1996.

\bibitem{Sau2024}
Federico Sau.
\newblock Concentration and local smoothness of the averaging process.
\newblock {\em Electronic Journal of Probability}, 29(none), January 2024.

\bibitem{Smoluchowski1927}
Marian Smoluchowski.
\newblock Drei vorträge über diffusion, brownsche molekularbewegung und koagulation von kolloidteilchen.
\newblock {\em Pisma Mariana Smoluchowskiego}, 2(1):530--594, 1927.

\bibitem{Tom2023}
T\^ania Tomé, Carlos~E. Fiore, and Mário~J. de~Oliveira.
\newblock Stochastic thermodynamics of opinion dynamics models.
\newblock {\em Physical Review E}, 107(6), June 2023.

\end{thebibliography}

\end{document}